\documentclass{article}
\usepackage{amsmath, amsthm, amssymb}
\usepackage[margin=0.8in]{geometry}
\usepackage{parskip}
\usepackage{hyperref}
\usepackage{graphicx} % Required for inserting images
\usepackage{xcolor}
\usepackage{subcaption}

\newenvironment{claimproof}[1][\proofname]{%
  \begin{proof}[#1]%
}{%
  \end{proof}%
}

\newcommand{\intI}{\mathcal{I}}

\newcommand{\vomega}{\vec{\omega}}
\newcommand{\vchi}{\vec{\chi}}

\newtheorem{lemma}{Lemma}
\newtheorem{cor}[lemma]{Corollary}
\newtheorem{theorem}[lemma]{Theorem}

\newtheorem{definition}[lemma]{Definition}
\newtheorem{obs}[lemma]{Observation}
\newtheorem{claim}[lemma]{Claim}

\title{Crossing tournaments are polynomially $\vec{\chi}$-bounded}

\usepackage{authblk}

\author[1]{Lila Crew}
\author[1]{Xinyue Fan}
\author[1]{Hidde Koerts}
\author[2]{Benjamin Moore}
\author[1]{Sophie Spirkl\thanks{Emails: (lcrew, xinyue.fan, hkoerts, sspirkl)@uwaterloo.ca, Ben.Moore@umanitoba.ca \\
 We acknowledge the support of the Natural Sciences and Engineering Research Council of Canada (NSERC), [funding reference numbers RGPIN-2020-03912 and RGPIN-2022-03093].
Cette recherche a \'et\'e financ\'ee par le Conseil de recherches en sciences naturelles et en g\'enie du Canada (CRSNG), [num\'eros de r\'ef\'erence RGPIN-2020-03912 et RGPIN-2022-03093]. This project was funded in part by the Government of Ontario. Benjamin Moore also acknowledges the support of NSERC grant RGPIN-2025-07125.  This research was conducted while Spirkl was an Alfred P. Sloan Fellow.  This research was undertaken, in part, thanks to
funding from the Canada Research Chairs Program.}
}

\affil[1]{University of Waterloo, Department of Combinatorics and Optimization, Waterloo, Canada}

\affil[2]{University of Manitoba, Department of Mathematics, Winnipeg, Canada}

\date{\today}

\begin{document}

\maketitle

\begin{abstract}
    Given a tournament $T$, Aboulker, Aubian, Charbit, and Lopes (2023) defined its clique number $\vec{\omega}(T)$ as the minimum clique number of a backedge graph of $T$, and raised the question: Which classes of tournaments are polynomially $\vec{\chi}$-bounded? Aboulker, Duron, Jacob, Kimbrough, Thomass\'{e}, and this work's authors (2026) showed that this holds for classes of tournaments whose arc sets may be written as the union of a bounded number of comparability digraphs.

    What about classes of tournaments that do not admit such a decomposition? The crossing tournaments of Nguyen, Scott, and Seymour (2025) are an example of such a class, as shown in the aforementioned 2026 work; we show that nonetheless crossing tournaments are polynomially $\vec{\chi}$-bounded by adapting a method of Davies and McCarty (2021) and Davies (2022).
    We additionally show that we cannot extend this result for crossing tournaments to tournaments with chordal graphs as backedge graphs.
\end{abstract}

\section{Introduction}
Throughout this work we use common graph theory terminology and notation as found in \cite{diestel}. We consider only finite and simple graphs, without loops and parallel edges. For a natural number $k$, we write $[k]$ for the set $\{1, \dots, k\}$.  A \textbf{proper $k$-colouring of a graph} $G$ is a function $f: V(G) \rightarrow [k]$ such that whenever $uv \in E(G)$, we have $f(u) \neq f(v)$. Given a proper $k$-colouring of $G$, we will refer to $f^{-1}(i)$ for $i \in [k]$ as a \textbf{colour class} (of $G$ with respect to $f$). The \textbf{chromatic number of} $G$ is the smallest $k$ such that $G$ has a proper $k$-colouring, and is denoted $\chi(G)$. A \textbf{clique} of $G$ is a set $A \subseteq V(G)$ such that for all $u,v \in A$, $uv \in E(G)$. The \textbf{clique number of} $G$ is the largest size of a clique of $G$, and is denoted $\omega(G)$.

It is straightforward to see that for any graph $G$, we have $\omega(G) \leq \chi(G)$, as any pair of vertices of a clique must receive different colours in a proper colouring. A graph $G$ is called \textbf{perfect} if for every induced subgraph $H$ of $G$, we have $\omega(H) = \chi(H)$. The celebrated Strong Perfect Graph Theorem was proved in 2002 and completely characterizes perfect graphs.

\begin{theorem}[Chudnovsky, Robertson, Seymour and Thomas \cite{strongperfect}]\label{thm:strongperfect}

A graph $G$ is perfect if and only if for all odd $n \geq 5$, the graph $G$ does not contain an $n$-vertex cycle or its complement as an induced subgraph. 
    
\end{theorem}

A next natural step up from perfect graphs is $\chi$-bounded graphs \cite{gyarfas1987problems}: A hereditary class $\mathcal{C}$ of graphs is said to be \textbf{$\chi$-bounded} if there exists a function $f: \mathbb{Z}^+ \rightarrow \mathbb{Z}^+$ such that for each $G \in \mathcal{C}$, we have $\chi(G) \leq f(\omega(G))$. The class $\mathcal{C}$ is further said to be \textbf{polynomially $\chi$-bounded} if $f$ may be chosen to be a polynomial in $\omega(G)$.

The problem of understanding which families of graphs are $\chi$-bounded is a very active area of study (see for example, \cite{cliquewidth,  bourneuf2025bounded, bousquetspaper, GroundedLgraphs,Lgraphs,survey}). One operation that plays well with $\chi$-boundedness is substitution of a graph $H$ in for a vertex $v$, which can be pictured as ``blowing up'' the vertex $v$ into a copy of $H$, and keeping adjacencies to $v$ as adjacencies to all vertices in $H$. Given a class $\mathcal{C}$ of graphs, the \textbf{substitution closure of $\mathcal{C}$}, notated as $\mathcal{C}^{subst}$, is the smallest collection of graphs containing $\mathcal{C}$ which has the property that substituting any graph in $\mathcal{C}^{subst}$ for a vertex of another (possibly identical) graph in $\mathcal{C}^{subst}$ at any vertex results in a graph in $\mathcal{C}^{subst}$ (that is, it is closed under substitution). The following result was proved in 2013.

\begin{theorem}[Chudnovsky, Penev, Scott, and Trotignon \cite{substitution}, Theorem 2.3] \label{thm:subst}

Let $\mathcal{C}$ be a polynomially $\chi$-bounded class of graphs. Then $\mathcal{C}^{subst}$ is also polynomially $\chi$-bounded.
    
\end{theorem}

In this paper, we study a generalization of $\chi$-boundedness to directed graphs as defined by Aboulker, Aubian, Charbit, and Lopes \cite{original}. A \textbf{directed graph} or \textbf{digraph} $D$ consists of a vertex set $V(D)$ and a set of arcs $A(D)$ that are \textit{ordered} pairs of vertices of $D$ (in particular, digraphs are simple in this paper, so no loops, digons, or multi-edges are allowed\footnote{Many authors call this an \textit{orientation} of a digraph, but we will use the term simple, as is more usual for undirected graphs}. We may write the arc $(v,w)$ of $D$ as any of $vw, v \rightarrow_D w, $ or $v \rightarrow w$. Given an arc $vw \in A(D)$, we say that $w$ is the \textbf{head} of the arc $vw$ and $v$ is the \textbf{tail} of the arc $vw$. For this paper, we will mostly be interested in studying tournaments. A \textbf{tournament} $T$ is a digraph in which for every pair of distinct vertices $v,w \in V(T)$, exactly one of $vw$ and $wv$ are in $A(T)$; equivalently, a tournament is an orientation of a complete graph. 

One tool for studying digraphs is to turn them back into ordered, undirected graphs. An \textbf{ordered graph} is a graph $G$ equipped with a total ordering $<_G$ of $V(G)$. Definitions used for graphs can also be used for ordered graphs by simply ignoring the ordering. One particularly useful ordered graph is the backedge graph which, amongst other things, encodes all arcs of a tournament. Given a digraph $D$ and a total ordering $<_B$ on $V(D)$, the \textbf{backedge graph of $D$ with respect to $<_B$}, notated $B(D, <_B)$ (or simply $B$ when $D$ and $<_B$ are clear from context) is the ordered graph with $V(B) = V(D)$, ordering $<_B$, and edge set $E(B) = \{vw: v <_B w, w \rightarrow_D v\}$. We say that a graph $G$ is a backedge graph of a digraph $D$ if there is an ordering $<_G$ of $V(G)$ such that the corresponding ordered graph is isomorphic to some backedge graph of $D$. Backedge graphs were first introduced by Berger, Choromanski, Chudnovsky, Fox, Loebl, Scott, Seymour, and Thomass\'{e}~\cite{BERGER20131} to study the colouring of tournaments excluding (induced) subtournaments. 

More specifically, in this paper we will focus on a notion of $\vec{\chi}$-boundedness in tournaments as defined in \cite{original}. We need to give the notions of chromatic number and clique number as they apply to tournaments. In the case of a digraph $D$, in 1982 Neumann-Lara \cite{neumann} defined the \textbf{dichromatic number} of $D$, notated $\vec{\chi}(D)$, as the minimum number $k$ such that there exists $f: V(D) \rightarrow [k]$ such that each $f^{-1}(i)$ for $i \in [k]$ induces an acyclic subdigraph; as in the case of undirected graphs, we may refer to each $f^{-1}(i)$ as a \textbf{colour class} (of $D$ with respect to $f$). We will frequently use the following elementary but fundamental result which further motivates backedge graphs:

\begin{lemma}[Folklore; see \cite{original}]\label{lem:minchromD}

Let $D$ be a digraph. Then the dichromatic number of $D$ is equal to the smallest chromatic number attained by a backedge graph of $D$. That is, 
$$\vec{\chi}(D) = \min_B \chi(B),$$ 
where the minimum ranges over all backedge graphs $B$ of $D$.
    
\end{lemma}

In 2023, motivated by Lemma \ref{lem:minchromD}, Aboulker, Aubian, Charbit, and Lopes \cite{original} introduced a notion of \textbf{clique number for a tournament $T$}, notated $\vec{\omega}(T)$, as $\min_B \omega(B)$, where the minimum ranges over all backedge graphs $B$ of $T$\footnote{One can define the diclique number of a general digraph analogously, but we will not need this for our purposes in this paper.}. Note that it follows from the definition and Lemma \ref{lem:minchromD} that for every tournament $T$, we have $\vomega(T) \leq \vchi(T)$. With this in hand, we can now generalize $\chi$-boundedness to digraphs. Given a hereditary class of digraphs $\mathcal{D}$, we say that the class is \textbf{$\vchi$-bounded} if there exists a function $f:\mathbb{Z}^{+} \to \mathbb{Z}^{+}$ such that for each $D \in \mathcal{D}$ we have $\vchi(D) \leq f(\vomega(D))$. As before, if $f$ can be taken to be a polynomial, then we say that $\mathcal{D}$ is \textbf{polynomially $\vchi$-bounded}.

Substitution generalizes to digraphs in the natural way, and the substitution closure of a class of digraphs is defined analogously to the graph case (we will not need a formal definition for this paper). One of the main results of Aboulker, Aubian, Charbit, and Lopes \cite{original} is that  substitution preserves $\vec{\chi}$-boundedness in the setting of tournaments, a result analogous to Theorem \ref{thm:subst} of \cite{substitution} on undirected graphs.\begin{theorem}[Aboulker, Aubian, Charbit, and Lopes, \cite{original}]\label{thm:nonpolysubst}

Let $\mathcal{C}$ be a $\vec{\chi}$-bounded class of tournaments. Then $\mathcal{C}^{subst}$ is $\vec{\chi}$-bounded.
    
\end{theorem}

However, in the case of undirected graphs, note that substitution also preserves the stronger condition of polynomial $\chi$-boundedness. This stronger notion has also been well-studied (see for example \cite{lindapaper,davies,GroundedLgraphs, nolongholevtxminor,Nesetril2020ClusteringPO}). In fact, it was not known for certain that these notions for undirected graphs were even distinct until a 2024 breakthrough by Bria\'nski, Davies, and Walczak \cite{separating} showed the existence of $\chi$-bounded classes of undirected graphs that are not polynomially $\chi$-bounded. It is thus natural to ask whether this stronger condition may also be true in tournaments; Aboulker et al.\  \cite{original} leave it as an open question in their work.

Aboulker, Duron, Jacob, Kimbrough, Thomass\'{e}, and this work's authors \cite{bigNote} showed recently that if a class of tournaments has bounded chromatic number, then the closure of this class under substitution is polynomially $\vec{\chi}$-bounded. 

\begin{theorem}[\cite{bigNote}]\label{thm:boundedchi}

   Let $\mathcal{C}$ be a class of tournaments such that each tournament in $\mathcal{C}$ has chromatic number at most $k$. Then for each $T \in \mathcal{C}^{subst}$, we have  $\vec{\chi}(T) \leq (\omega(T))^{2\lceil \log_2(k) \rceil +1}$. In particular, the class $\mathcal{C}^{subst}$ is polynomially $\vec{\chi}$-bounded. 
\end{theorem}

Graphs and digraphs arising from posets are used in proving Theorem \ref{thm:boundedchi}. A \textbf{poset} $P$ consists of a set $V(P)$ equipped with a partial order $<_P$, meaning that $<_P$ is reflexive, antisymmetric, and transitive. A digraph $D$ is a \textbf{comparability digraph} if there exists a poset $P$ with $V(P) = V(G)$ such that for distinct $v,w \in V(D)$, we have $vw \in A(D)$ if and only if $v <_P  w$. Equivalently, a comparability digraph is one where whenever $a,b,c \in V(D)$ are such that $a \rightarrow_D b$ and $b \rightarrow_D c$, then $a \rightarrow_D c$. A graph $G$ is a \textbf{comparability graph} if $G $ is the underlying undirected graph for a comparability digraph $D$; equivalently, there exists a poset $P$ with $V(P) = V(G)$ such that for any distinct $v,w \in V(G)$, we have $vw \in E(G)$ if and only if $v$ and $w$ are comparable in $P$ (meaning either $v <_P w$ or $w <_P v$). 

Theorem \ref{thm:boundedchi} relies in part on an intermediate result that every tournament in a class with bounded chromatic number decomposes into a bounded number of comparability digraphs. A digraph $D$ is said to \textbf{decompose into} (or \textbf{has a decomposition into}) digraphs $D_1,\dots,D_k$ if $V(D) = V(D_1) = \dots = V(D_k)$ and $A(D_1) \cup \dots \cup A(D_k) = A(D)$. Analogously, a graph $G$ is said to \textbf{decompose into} graphs $G_1, \dots, G_k$ if $V(G) = V(G_1) = \dots = V(G_k)$ and $E(G_1) \cup \dots \cup E(G_k) = E(G)$. We say that such a decomposition is \textbf{disjoint} if the edge sets (respectively, arc sets) are pairwise disjoint. 

This decomposition is combined with the easily verifiable observation (Lemma 6 of \cite{bigNote}) that every backedge graph of a comparability digraph is a comparability graph, as well as the following well-known result of Berge \cite{berge}: 
\begin{lemma}[Berge \cite{berge}]\label{lem:compperf}
Comparability graphs are perfect.
\end{lemma}

This leads to the natural question: Could it be the case that decomposing into a bounded number of comparability digraphs is necessary for polynomial $\vec{\chi}$-boundedness in tournaments? We show that the answer to this is no, using as a counterexample a class of tournaments defined by having backedge graphs that arise from interval systems.

An \textbf{interval system} $\mathcal{I}$ is a set of open intervals in $(0,1)$ such that no two elements of $\mathcal{I}$ share an endpoint. Given an interval $I \in \mathcal{I}$, its left endpoint is notated $L(I)$, and its right endpoint is notated $R(I)$. The interval system $\mathcal{I}$ is naturally equipped with two (not necessarily distinct) total orders: $<_L$, the ordering by left endpoint, and $<_R$, the ordering by right endpoint. There is also a natural partial order $<_\intI$ given by $I <_\intI J$ if and only if $R(I) < L(J)$ (that is, the intervals are disjoint and $I$ is ``before'' $J$ in $(0, 1)$).

The \textbf{interval graph} $G_{\mathcal{I}}$ of an interval system $\mathcal{I}$ is an ordered graph with $V(G_{\mathcal{I}}) = \mathcal{I}$, ordered according to $<_L$ (that is, by left endpoint), and $E(G_{\intI})$ is the set of all pairs of intervals $I, J$ with $L(I) < L(J)$ such that $R(I) > L(J)$; that is, two intervals are adjacent in $G_{\mathcal{I}}$ if and only if they have a non-empty intersection (as intervals).
\begin{figure*}
    \centering
        \begin{subfigure}{0.29\textwidth}
            \centering
            \includegraphics[width=\linewidth]{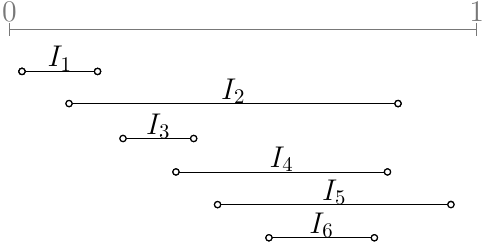}
        \end{subfigure}
        \hfill
        \begin{subfigure}{0.29\textwidth}
            \centering
            \includegraphics[width=\linewidth]{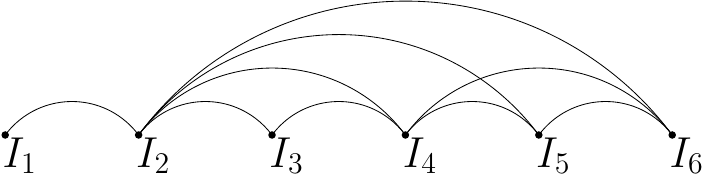}
        \end{subfigure}
        \hfill
        \begin{subfigure}{0.29\textwidth}
            \centering
            \includegraphics[width=\linewidth]{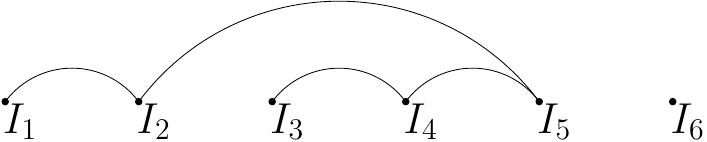}
        \end{subfigure}
    \caption{An interval system $\mathcal{I}$, labelled according to $<_L$ (left); its corresponding ordered interval graph $G_{\mathcal{I}}$ (middle); and its corresponding ordered overlap graph (right). }
    \label{fig:placeholder}
\end{figure*}

Gy{\'a}rf{\'a}s, Marits, and T{\'o}th showed in 2024 that interval graphs cannot be decomposed into a bounded number of comparability graphs.

\begin{theorem}[Gy{\'a}rf{\'a}s, Marits, and T{\'o}th \cite{perfectIntoComparability}, Theorem 8 restatement]\label{thm:gyarfas}
For each $k$, there exists an interval graph that cannot be decomposed into fewer than $k$ comparability graphs.
\end{theorem}

A tournament $T$ is a \textbf{crossing tournament} if it admits a backedge graph which is an interval graph (with the ordering $<_L$). The class of crossing tournaments was first considered by Nguyen, Scott, and Seymour in 2025 as an example of a class of tournaments with unbounded clique number \cite{crossing}.

In analogy to the above, in \cite{bigNote} it is shown that crossing tournaments do not generally decompose into a bounded number of comparability digraphs, using ideas from the proof of Theorem \ref{thm:gyarfas} of \cite{perfectIntoComparability}.

\begin{theorem}[\cite{bigNote}]\label{thm:notkcomp}

For every fixed positive integer $k$, there is a crossing tournament $T$ in which cannot be decomposed into fewer than $k$ comparability digraphs. 
    
\end{theorem}

We show as our main result that the class of crossing tournaments is polynomially $\vec{\chi}$-bounded, giving an explicit example of such a family that is not expressible as a union of a bounded number of comparability digraphs.

\begin{theorem}\label{thm:finaltheorem}

Let $T$ be a crossing tournament, and let $\omega(T) = \omega$. Then
$$
\vec{\chi}(T) \leq 15\omega^4+2\omega^3\log_2(\omega)+6\omega^3.
$$
    
\end{theorem}

It is natural to see if we can strengthen the above result to a wider class of backedge graphs than just interval graphs. In particular, every interval graph is \textbf{chordal}, which means that there are no induced cycles of length more than three. Here, we show that the answer is negative, the class of tournaments that contain a chordal graph as a backedge graph is not $\vec{\chi}$-bounded. Thus it is not obvious how to extend Theorem \ref{thm:finaltheorem} to a wider class of tournaments.

\begin{theorem}
\label{thm:chordalexample}
    For every integer $k$, there exists a tournament $T$ and ordering $<$ such that $B(T,<)$ is chordal, with $\vec{\omega}(T) \leq 2$ and $\vec{\chi}(T) \geq k$. 
\end{theorem}

We give a brief outline of the proof of Theorem \ref{thm:finaltheorem} here. Letting $T$ be any crossing tournament on an interval system $\intI$, we select a backedge graph $G'$ realizing the clique number of $T$, meaning $\omega(G') = \vomega(T)$. The graph $G'$ is an ordered graph whose vertices are intervals, with ordering $<_{G'}$ not necessarily equivalent to an ordering by left endpoint. It is enough to find a polynomial function $f$ so that for any such $G'$ we have $\chi(G') \leq f(\omega(G'))$.

We then reduce this problem to finding such a polynomial function for a subgraph $H$ of $G'$ that is a (not necessarily induced) subgraph of the \textbf{overlap graph of $\intI$}, whose vertex set is $\intI$ ordered by left endpoint and whose edges are given by pairs $I, J$ of intervals in $\intI$ with $L(I) < L(J) < R(I) < R(J)$; that is, the edges of the overlap graph of $\intI$ are the pairs of intervals which overlap. Overlap graphs are also known as circle graphs (the intersection graphs of chords on a circle)\footnote{Incidentally, the overlap graph of $\intI$, as an unordered graph, is the same as the backedge graph of the crossing tournament $T$ corresponding to $\intI$ with the ordering $<_R$ given by right endpoints of the intervals.}. The $\chi$-boundedness of circle graphs was originally established by Gy{\'a}rf{\'a}s in 1985 \cite{gyarfasCircle}. Resolving a long-standing question, Davies and McCarty showed in 2021 \cite{mccarty} that the class of circle graphs is polynomially $\chi$-bounded by an $O(\omega^2)$ function, followed in 2022 by an improvement of Davies \cite{davies} to an $O(\omega \log \omega)$ function that is best possible up to a constant factor (Kostochka \cite{kostochka}).

To bound the chromatic number of $H$, we closely follow the proof in \cite{davies}, making adjustments as needed. The general idea is to take an ordered sequence $\mathcal{P}$ of \textbf{pillars} $p_1,p_2,\dots, p_r$, or points in $(0,1)$, and place them one at a time, \textbf{assigning to the pillar $p_i$} all intervals of $\intI_H$ that contain $p_i$ and have not previously been assigned to a pillar. We let the set of intervals assigned to pillar $p_i$ be given by $A_{p_i}$, and we call $\mathcal{P}$ a \textbf{pillar system}.

Each time we place a pillar, we then give a nice colouring $\phi_{p_i}$ (in a sense explained precisely in Section \ref{sec:onepillar}) to the intervals of $A_{p_i}$. In particular, we will choose $\phi_{p_i}$ in such a way that $\phi_{p_1} \cup \dots \cup \phi_{p_i}$ is a proper partial colouring of $H$ determined uniquely by $H$ and the choices of $p_1, \dots, p_i$, and such that $\phi_{p_i}$ uses at most $\omega(H)$ colours that were previously unused in the process. Ultimately, we show that there is a way to choose $\mathcal{P}$ such that the resulting colouring uses a number of colours at most polynomial in $\omega$, from which Theorem \ref{thm:finaltheorem} follows.

It is natural to ask why we cannot just use the exact same proof from \cite{davies}. The main issue is the critical Lemma 10 of \cite{davies}, Davies' Tur\'{a}n-type result for bounding ``how bad things can get'' during the process of adding pillars and colouring. If this lemma worked in our setting, the result would follow quickly; unfortunately, it does not. For one, it requires Lemma 3 of the same paper, a technical result about simplifying colourings when all intervals contain one pillar, and the natural extension of this lemma to our setting does not hold in its entirety. Additionally, the end of Davies' Lemma 10 results in a contradiction formed by pairwise overlapping intervals assigned to different pillars forming a clique that is too large; in our setting, since we use only a subgraph of the circle graph, we also require additional relations to hold to form a clique, which do not appear feasible to analyze across intervals assigned to different pillars.

However, we are able to bypass this obstacle by modifying the method in \cite{davies} for placing pillars to ensure that we leave uncoloured intervals between pillars to take advantage of the structure of our graph, while still being able to colour those intervals at the end with few colours.

\section{Crossing Tournaments are Polynomially $\vec{\chi}$-bounded}

We show that the class of crossing tournaments is polynomially $\vec{\chi}$-bounded. Let $\intI$ be an interval system, and let $G$ be its corresponding ordered graph with total order $<_G$ (corresponding to the total ordering of left endpoints of intervals in $\intI$). Let $T$ be the tournament admitting $G$ as a backedge graph. We aim to give a function bounding $\vchi(T)$ by a polynomial in $\vomega(T)$. Let $G'$ be a particular backedge graph of $T$ such that $\omega(G') = \vomega(T)$. Let $<_{G'}$ be the total ordering of $\intI$ used by the ordered graph $G'$. It suffices to bound $\chi(G')$ by a polynomial in $\omega(G')$, since $\vchi(T) \leq \chi(G')$ by Lemma \ref{lem:minchromD}.

\subsection{Reducing to a Subgraph of a Circle Graph}

We first note that the edges $IJ$ of $G'$ can be partitioned into three sets depending on whether the intervals are disjoint, one is fully contained in the other, or the intervals overlap. The three sets are:
\begin{itemize}
    \item Type 1 edges ($I, J$ disjoint): $I >_{G'} J$ and $R(I) < L(J)$. 
    \item Type 2 edges ($J \subseteq I$): $I <_{G'} J$ and $R(I) > R(J)$ and $L(I) < L(J)$.
    \item Type 3 edges ($I$ and $J$ overlap): $I <_{G'} J$ and $L(I) < L(J) < R(I) < R(J)$. 
\end{itemize}
\begin{figure}
    \centering
    \includegraphics[width=0.9\linewidth]{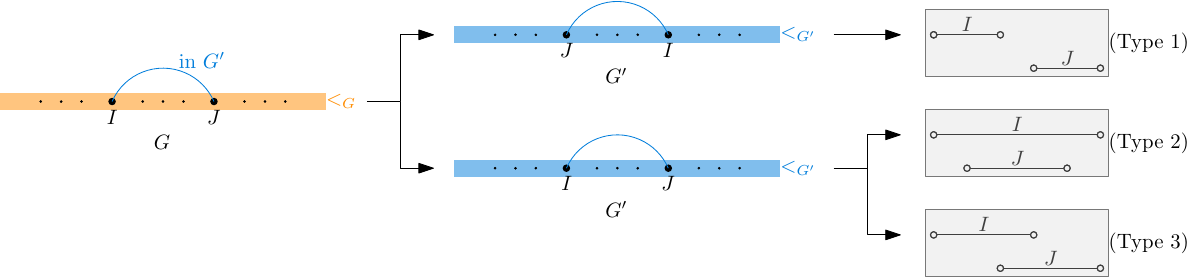}
    \caption{Every edge $IJ$ of $G'$ with $L(I) < L(J)$ is of one of these three types.}
    \label{fig:types}
\end{figure}

See Figure \ref{fig:types}. Let $E_i$ refer to the set of type $i$ edges of $G'$ for $i \in \{1,2,3\}$.

\begin{lemma}\label{lem:G1G2}

Let $G', E_1, E_2$ be as above. Then the subgraph $G_{12}' = (V(G'),E_1 \cup E_2)$ of $G'$ satisfies $\chi(G_{12}') \leq (\omega(G'))^2$.

\end{lemma}

\begin{proof}

The subgraph $G_1 = (V(G'), E_1)$ consisting of only type $1$ edges is a comparability graph on $V(G')$ with respect to the partial order $<_1$ defined by $I <_1 J$ if and only if  $I >_{G'} J$, and $I <_{\intI} J$. Likewise, the subgraph $G_2 = (V(G'), E_2)$ consisting of only type $2$ edges forms a comparability graph on $V(G')$ with respect to the partial order $<_2$ defined by $I <_2 J$ if and only if  $I <_{G'} J$ and $J \subseteq I$.

Thus each of $G_1$ and $G_2$ is perfect by Lemma \ref{lem:compperf}, and so $\chi(G_i) = \omega(G_i) \leq \omega(G')$ for $i \in \{1,2\}$. Therefore, letting $g_i: V(G') \rightarrow [\omega(G')]$ be a proper colouring of $G_i$ for $i \in \{1,2\}$, we may properly colour $G_{12}'$ via the product colouring $g: V(G') \rightarrow [\omega(G')]^2$ given by $g(v) = (g_1(v), g_2(v))$.
\end{proof}

Based on this, we can reduce the problem as follows: 

\begin{obs}\label{obs:reduction}
Let $T$ be a crossing tournament, and let $G'$ be a backedge graph of $T$ realizing clique number. Let $H_1, \dots, H_{\omega(G')^2}$ be the subgraphs of $G'$ induced by restricting $G'$ to the vertices of one of the colour classes of an $(\omega(G'))^2$-colouring of $(V(G'), E_1 \cup E_2)$ as given by Lemma \ref{lem:G1G2} (noting that these are not proper colour classes on all of $G'$ since we have type 3 edges). If we colour each of the $H_i$ with a separate set of colours, we obtain a colouring of $G'$. 

Thus, if there exists a non-decreasing universal polynomial $f$ such that for all $i$ we have $\chi(H_i) \leq f(\omega(H_i))$, then since for all $i$ we have $\omega(H_i) \leq \omega(G')$ it follows that
\begin{align*}
 \vchi(T) \leq \chi(G') &\leq \sum_{i=1}^{(\omega(G'))^2} \chi(H_i) \\ &\leq (\omega(G'))^2 f(\max_i (\omega(H_i)) \\ &\leq (\omega(G'))^2 f(\omega(G')) \\&= (\vomega(T))^2f(\vomega(T)).
\end{align*}

\end{obs}

Thus, we need only show that such an $f$ exists. From now on, we will let $H$ be a subgraph of $G'$ induced by one of the colour classes of a colouring given by Lemma \ref{lem:G1G2} (noting that the vertex set of $H$ is thus a subset of the interval system $\intI$ defining $T$). It follows that all edges of $H$ are of type 3. 

\subsection{Polynomially $\chi$-Bounding ``Local Circle" Graphs}

We start by summarizing the necessary aspects of the subgraph $H$ of $G'$ we will be working with from now on, and combine them into a definition of $H$ that we can use independently of the notation from the previous section. In particular, the last two statements in the following definition are due to the fact that we have no edges of types 1 and 2 in $H$. 

For clarity, let $\intI_H = V(H)$ denote the intervals from $\intI$ that are vertices of $H$; we will use this notation when we wish to emphasize the structure of the vertex set as an interval system. In the following definition, $<_H$ plays the role of the restriction of the total ordering $<_{G'}$ to $\intI_H$.

\begin{definition}\label{def:H}
Let $H$ be an ordered graph whose vertex set is an interval system $\intI_H$ with ordering $<_H$ (not necessarily an ordering by left endpoint). We say that $H$ is a \textbf{crossing graph} if $E(H)$ is equal to the set of pairs of intervals $I, J$  such that
\begin{enumerate}
    \item[(a)] $I <_H J$, and
    \item[(b)] $L(I) < L(J) < R(I) < R(J)$, 
\end{enumerate}
and furthermore, the following statements about $<_H$ hold for any pair of intervals $I,J \in \intI_H$ where $L(I) < L(J)$:
\begin{itemize}
    \item[(i)] Whenever $R(I) < L(J)$, we have $I <_H J$.
    \item[(ii)] Whenever $R(I) > R(J)$, we have $I >_H J$.
\end{itemize}
\end{definition}

Unlike with the other edge types from the previous section, the edges of $H$ may not form a comparability graph, because the relation $L(J) < R(I)$ is not generally transitive.   However, the graph with vertex set $\intI_H$ and edges defined only by (b) is the overlap graph corresponding to $\intI_H$, so $H$ is a subgraph of that overlap graph with some restrictions. Therefore, we may give the polynomial function for $H$ that we want in this setting using an argument closely following that of a result of Davies~\cite{davies} on circle graphs (with some critical modifications). 

\subsubsection{Subgraphs of $H$ Assigned to One Pillar}\label{sec:onepillar}

Recall that a pillar is a point in $(0,1)$, and that the idea (as in~\cite{davies}) is to assign intervals of $\intI_H$ to pillars that they contain. We first investigate what happens when one pillar suffices, meaning that all intervals of $\intI_H$ contain the same point. It is well-known (see, for example, \cite{davies}) that the overlap graph of an interval system $\intI_H$ in which each interval contains some fixed point $p \in (0,1)$ is a permutation graph, or equivalently, a graph that is both a comparability graph and a co-comparability graph~\cite[Proposition 4.7.1]{brandstadt1999graph}. The graph $H$ is only a subgraph of a circle graph, so in general the corresponding analogue is not a permutation graph. However, we show that it is still a comparability graph, and hence perfect by Lemma~\ref{lem:compperf}. This suffices for our purposes.

\begin{lemma}\label{lem:compgraph}  Let $H$ be a crossing graph. Let $X \subseteq V(H)$ and $ p \in (0, 1)$ such that every interval in $X$ contains $p$. Then $H[X]$ is a comparability graph, and hence $\omega(H)$-colourable. 
\end{lemma}

\begin{proof}
    Given $I,J \in X$, we have $L(I), L(J) < p < R(I), R(J)$; consequently, we have $IJ \in E(H)$ if and only if
    \begin{itemize}
        \item $L(I) < L(J)$; 
        \item $R(I) < R(J)$; and
        \item $I <_H J$. 
    \end{itemize}
    This defines a partial order, and hence $H[X]$ is a comparability graph. It follows that $H[X]$ is perfect by Lemma \ref{lem:compperf}, and hence $\omega(H)$-colourable. 
\end{proof}

In the setting of Lemma \ref{lem:compgraph}, Davies \cite{davies} uses a particularly well-structured colouring, which allows optimizing the bound on the number of colours needed later in the proof. While we could define a similar colouring in our setting, we use a different strategy to bound the number of colours which does not rely on the choice of colouring. Therefore, we forego picking a particular colouring. 
\subsubsection{Placing Multiple Pillars}\label{sec:multipillar}

Our process for colouring $H$ proceeds as follows. We start with the interval $(0, 1)$ and an empty set of pillars. Iteratively, we build a sequence $\mathcal{P} = p_1, \dots, p_r$ of points in $(0, 1)$ called \textbf{pillars}. For each pillar $p_i$, we denote by $A_{p_i}$ the set of intervals in $\mathcal{I}_H$ which contain $p_i$ and do not contain any of the pillars $p_1, \dots, p_{i-1}$. We call the set $A_{p_i}$ the \textbf{intervals assigned to $p_i$}. 

\begin{figure}
    \centering
    \includegraphics[width=0.5\linewidth]{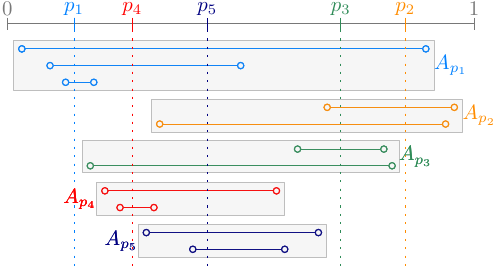}
    \caption{Pillars and the intervals assigned to them. For the arch $K = (p_5, p_3)$, we could choose $K_{p_1} = (0, p_1)$; $K_{p_4} = (p_1, p_4)$; $K_{p_5} = (p_4, p_5)$; $K_{p_3} = (p_3, p_2)$ and $K_{p_2} = (p_2, 1)$ in Lemma \ref{lem:pillarcontinuity}.} 
    \label{fig:placeholder}
\end{figure}

Whenever we add a pillar, we also construct a colouring $\phi_{p_i}$, with the property that together, $\phi_{p_1}, \dots, \phi_{p_r}$ form a colouring $\phi_{\mathcal{P}}$ of the subgraph of $H$ induced by all intervals that have been assigned to a pillar.  Suppose that we add a pillar $p_{r+1}$. We describe how to construct $\phi_{p_{r+1}}$. Let $F_{p_{r+1}}$ be the largest open interval containing $p_{r+1}$ which does not contain any pillars $p_j$ with $j < r+1$. In particular, all intervals in $A_{p_{r+1}}$ necessarily have both endpoints in $F_{p_{r+1}}$ (with one each on either side of $p_{r+1}$). 

Let $C$ be the set of the smallest $\omega$ positive integers that are not used by any $\phi_{p_j}$ with $j \leq r$ to colour any interval with an endpoint in $F_{p_{r+1}}$; the set of such intervals is a superset of the intervals that are adjacent in $H$ to some interval of $A_{p_{r+1}}$. Then define $\phi_{p_{r+1}} : A_{p_{r+1}} \rightarrow C$ to be a colouring of $H[A_{p_{r+1}}]$ using at most $\omega$ colours, as promised by Lemma \ref{lem:compgraph}. 

We now provide further definitions from \cite{davies} that we will need to proceed. We denote by $\phi_{\mathcal{P}}$ the partial proper colouring of $H$ given by $\bigcup_{p_i \in \mathcal{P}} \phi_{p_i}$. We denote by $\chi^*(\mathcal{P})$ the number of colours used by $\phi_{\mathcal{P}}$. Given a pillar system $\mathcal{P}$, an \textbf{arch} $K$ \textbf{of} $\mathcal{P}$ is a maximal open interval in $(0,1)$ containing no pillars in $\mathcal{P}$. In particular, this means that the endpoints of an arch are among $\{0,1\} \cup \mathcal{P}$.

Given an arch $K$, we will use the notation $\intI_H(K)$ to denote the set of intervals in $H$ that have \textit{exactly one} endpoint in $K$; that is, $\intI_H(K)$ consists of intervals coloured by $\mathcal{P}$ that might interfere later with colouring intervals contained strictly in $K$. Let $\intI_H(K,p_i)$ denote the set of intervals with exactly one endpoint in $K$ that are assigned to pillar $p_i$.

We restate the following auxiliary lemma from \cite{davies} for our usage; since the lemma depends only on $\intI_H$ and $\mathcal{P}$ and not on the adjacencies defining $H$, the proof is identical and so will not be restated. See also Figure \ref{fig:placeholder}. 

\begin{lemma}[Davies, \cite{davies}, Lemma 3.2]\label{lem:pillarcontinuity}

Given a crossing graph $H$, a pillar system $\mathcal{P}$ for $\intI_H$, and an arch $K$ of $\mathcal{P}$, there exist disjoint open intervals $\{K_{p_i}: p_i \in \mathcal{P}\}$ in $(0,1)$ such that each $I \in \intI_H(K,p_i)$ has one endpoint in $K$ and the other in $K_{p_i}$. 
    
\end{lemma}

Informally, this means that the not-in-$K$ endpoints of intervals in $\intI_H(K)$ corresponding to different pillars do not interleave. We require the following definition from \cite{davies}.
\begin{definition}\label{def:degree}

Let $H$ be a crossing graph, and let $\mathcal{P}$ be a pillar system for $\intI_H$. Let $K$ be an arch of $\mathcal{P}$, and let $J$ be any interval (not necessarily in $\intI_H$) with both endpoints in $K$. Then the \textbf{degree of} $J$ \textbf{with respect to} $\mathcal{P}$, notated as $\deg_{\mathcal{P}}(J)$, is equal to the number of colours of $\phi_{\mathcal{P}}$ used by intervals with an endpoint in $J$.
    
\end{definition}

Note that if $J \in \intI_H$, this is \textit{not} the same as the degree of $J$ as a vertex in $H$ (and our arguments will not involve vertex degrees); we generally use Definition \ref{def:degree} in the context of considering placing new pillars, and the new arches this would create. The following lemma from \cite{davies} (very slightly generalized for our purposes: Davies \cite{davies} works with a uniform bound $t$ on the degrees of intervals $J$, whereas we track how much the degree increases for each individual interval) holds exactly as proven there.

\begin{lemma}[Davies, \cite{davies}, Lemma 3.4]\label{lem:inductpillars}

Let $\mathcal{P}$ be a pillar system for $\intI_H$. Let $K$ be an arch of $\mathcal{P}$, and let $Q$ be a finite collection of pillars in $K$. Then there exists a pillar system $\mathcal{P}^*$, formed by adding the pillars of $Q$ in a specified order after those of $\mathcal{P}$, such that

\begin{itemize}
    \item $\chi^*(\mathcal{P}^*) \leq \max(\chi^*(\mathcal{P}),\deg_\mathcal{P}(K)+\omega(H)\lceil \log_2(|Q|+1)\rceil)$, and
    \item For each interval $J$ of $K \backslash Q$, we have $\deg_{\mathcal{P}^*}(J) \leq \deg_{\mathcal{P}}(J) +\omega(H)\lceil \log_2(|Q|+1)\rceil$.
\end{itemize}
    
\end{lemma}

Finally, we need one more lemma of \cite{davies}. Given positive integers $n$ and $m$, the \textbf{strong dominance poset $P_{n,m}$} has vertex set $[n] \times [m]$ and partial order $<_{P_{n,m}}$ satisfying that $(i_1,j_1) <_{P_{n,m}} (i_2,j_2)$ if and only if $i_1 < i_2$ and $j_1 < j_2$. A \emph{chain} in a poset is a set of pairwise comparable elements. 

\begin{lemma}[Davies, \cite{davies}, Theorem 4.2]\label{lem:strongposet}

Let $a,n,m$ be positive integers with  $a \leq \min(n,m)$. Let $P_{n,m}$ be the strong dominance poset with vertex set $[n] \times [m]$. If $S \subseteq [n] \times [m]$ does not contain the elements of a chain  in $P_{n,m}$ of length greater than $a$, then \\ $$|S| \leq a(n+m-a).$$
    
\end{lemma}

As noted in the introduction, if Davies' Lemma 4.3 \cite{davies} worked in our setting, we would be done; but in our case, adjacencies also depend on the order $<_H$, and therefore we need a modified approach. We will need one structural result that is a direct corollary of how the graph $H$ is defined (see Figure \ref{fig:28}). 
\begin{figure}
    \centering
    \includegraphics[width=0.5\linewidth]{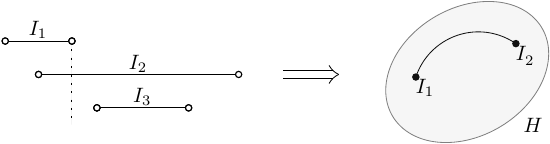}
    \caption{Lemma \ref{lem:nothree}.}
    \label{fig:28}
\end{figure}

\begin{lemma}\label{lem:nothree}
    Let $H$ be a graph and $\intI_H$ its interval system as defined at the beginning of this section. Suppose that there exist three intervals $I_1,I_2,I_3 \in \intI_H$ such that $L(I_1) < L(I_2) < R(I_1) < L(I_3) < R(I_3) < R(I_2)$. Then $I_1$ is adjacent to $I_2$ in $H$.
\end{lemma}

\begin{proof}
    By Definition \ref{def:H}, we have $I_2 >_H I_3$ and $I_1 <_H I_3$. Together, these imply that $I_1 <_H I_2$, so since also $L(I_1) < L(I_{2}) <R(I_1) < R(I_2)$, we have $I_1$ adjacent to $I_2$ in $H$.
\end{proof}

Now we are ready for the main result of this section.

\begin{theorem}\label{thm:mainresult}

Let $H$ be a crossing graph with corresponding interval system $\intI_H$ and total order $<_H$. Let $\mathcal{P}$ be a pillar system for $\intI_H$ with the following properties:
\begin{enumerate}
    \item[(a)] $\chi^*(\mathcal{P}) \leq 15\omega(H)^2+2\omega(H)\log_2(\omega(H))+5\omega(H)$.
    \item[(b)] Each arch of $\mathcal{P}$ contains at least one interval of $\intI_H$.
    \item[(c)] Each arch of $\mathcal{P}$ that contains a pair of disjoint intervals of $\intI_H$ has degree no greater than $15\omega(H)^2$.
\end{enumerate}

Suppose that there exists an arch $K$ of $\mathcal{P}$ containing at least one pair of disjoint intervals of $\intI_H$.

Then it is possible to extend $\mathcal{P}$ via pillars placed in $K$ to a pillar system $\mathcal{P}'$ so that $\mathcal{P}'$ satisfies the above three conditions \textbf{and} there exist two disjoint intervals  which are contained in the same arch of $\mathcal{P}$, but are not contained in the same arch of $\mathcal{P}'$. 
\end{theorem}

The key difference between this and Davies' Theorem 5.1 is that in our construction, we guarantee that there is at least one interval between each pair of consecutive pillars, so that we can take advantage of Lemma \ref{lem:nothree} (where these intervals will play the role of $I_3$). In exchange, our bound is slightly worse than that of Davies, but we still easily retain polynomial $\vec{\chi}$-boundedness. The basic proof technique will still be similar, but we stop when for each arch, all intervals contained within it have a point in common.

\begin{proof}[Proof of Theorem \ref{thm:mainresult}]
    We perform the following algorithm until it terminates: at the $i^{th}$ step, let $q_{i-1}$ be the most recently placed pillar (with $q_0 = L(K)$). Find the leftmost point $x$ in $K$ (if it exists) such that \textbf{both}
    \begin{itemize}
        \item $\deg_{\mathcal{P}}((q_{i-1},x)) \geq 10\omega(H)^2$ \textbf{and}
        \item There is at least one interval of $\intI_H$ entirely contained in $(q_{i-1},x)$.
    \end{itemize}
    If such an $x$ is found, set $q_i = x$ and continue. Otherwise, we are finished placing new pillars in $K$. Let $q_n$ be the last pillar placed. Note that by construction, each interval $(q_{i-1},q_i)$ for $i \in\{ 1,\dots,n\}$ either has degree exactly equal to $10\omega(H)^2$, or contains exactly one interval of $\intI_H$. Furthermore, the interval $(q_n,r(K))$ either has degree less than $10\omega(H)^2$, or contains no intervals of $\intI_H$.

     Let $J_i = (q_{i-1},q_i)$ for $i \in \{1, \dots, n-1\}$, and let $J_n = (q_{n-1}, R(K))$. Let $s = \sum_{i=1}^{n} \deg_{\mathcal{P}}(J_i)$. By construction it is immediate that $s \geq 10n\omega(H)^2$. Note that $\deg_{\mathcal{P}}(J_n) \geq \deg_{\mathcal{P}}((q_{n-1}, q_n))$ since $(q_{n-1}, q_n) \subseteq J_n$. 
     
      For each pillar $p \in \mathcal{P}$, letting $K_p$ be the open interval of $(0,1) \backslash K$ corresponding to the pillar $p$ as defined by Lemma \ref{lem:pillarcontinuity}, we may give a total order $<_N$ to the pillars of $\mathcal{P}$ as in \cite{davies} by defining $p_1 <_N p_2$ if and only if either $p_1$ and $p_2$ are on the same side of $K$ and $R(K_{p_1}) < L(K_{p_2})$, or $R(K_{p_2}) < L(K) < R(K) < L(K_{p_1})$.

    Let $\mathcal{J} = \{J_1, \dots, J_n\}$, and define $\mathcal{K}$ to be the set of all pillars in $\mathcal{P}$ that have been assigned an interval of $\intI_H$ with at least one endpoint in $K$. We write $\mathcal{K} = \{p_1, \dots, p_m\}$, listed in order with respect to $<_N$. Let $S' \subseteq [n] \times [m]$  consist of those ordered pairs $(a,b)$ such that there is an interval of $\intI_H$ assigned to $p_a$, with one endpoint in $K_{p_a}$ and the other endpoint in $J_b$. Note that since each pillar has at most $\omega(H)$ distinct colours among intervals assigned to it, we see that $\omega(H)|S'| \geq s$. Now, we proceed with the first substantial step. 

    \begin{claim}\label{claim:4omega+1}

    $S'$ does not contain the elements of a chain of size $4\omega(H)+1$ from the strong dominance poset $P_{n,m}$ with vertex set $[n] \times [m]$.
        
    \end{claim}

    \begin{claimproof}[Proof of Claim~\ref{claim:4omega+1}]

     Suppose otherwise, and let $C$ be such a chain. Let $\mathcal{I}_C$  be a set of intervals so that each interval in $\mathcal{I}_C$ corresponds to exactly one element of $C$. Then either there are at least $2\omega(H)+1$ intervals of $\mathcal{I}_C$ with one endpoint to the left of $K$, or at least that many intervals with one endpoint to the right of $K$. 

    Suppose first we have $2\omega(H)+1$ elements of $\mathcal{I}_C$ with one endpoint to the left of $K$. Let these be labelled $I_1,\dots,I_{2\omega(H)+1}$ so that $L(I_1) < L(I_2) < \dots < L(I_{2\omega(H)+1})$ and $R(I_1) < R(I_2) < \dots < R(I_{2\omega(H)+1})$. Furthermore, let $a_1 < \dots < a_{2\omega(H)+1}$ be positive integers so that $R(I_i) \in J_{a_i}$ for each $i$. For each $i \in [2\omega(H)-1]$, note that $R(I_i) < L(J_{a_{i+1}}) <  R(J_{a_{i+1}}) < R(I_{i+2})$, and that by construction $J_{a_{i+1}}$ contains at least one interval $I'$ of $\intI_H$. Therefore, $I_i, I_{i+2}, I'$ satisfy the conditions of Lemma~\ref{lem:nothree}, implying that $I_i$ is adjacent to $I_{i+2}$ in $H$, also implying that $I_i <_H I_{i+2}$. Since this holds for every $i \in [2\omega(H)-1]$, we have that $I_1 <_H I_3 <_H \dots <_H I_{2\omega(H)+1}$, and all of these intervals are pairwise overlapping, giving a clique of size $\omega(H)+1$ in $H$, a contradiction.

    Now, suppose instead that we have $2\omega(H)+1$ elements of $\mathcal{I}_C$ with one endpoint to the right of $K$. As above, let these be labelled $I_1,\dots,I_{2\omega(H)+1}$ so that $L(I_1) < L(I_2) < \dots < L(I_{2\omega(H)+1})$ and $R(I_1) < R(I_2) < \dots < R(I_{2\omega(H)+1})$, and let $a_1 < \dots < a_{2\omega(H)+1}$ be positive integers so that $R(I_i) \in K_{p_{a_i}}$ for each $i$. Note that for $i \in [2\omega(H)+1]$, we have that necessarily $p_{a_i} \in I_i$, and $p_{a_j} \notin I_i$ for all $j > i$ (since $p_{a_j}$ is placed before $p_{a_i}$). Therefore, for each $i \in [2\omega(H)-1]$, we have $R(I_i) < p_{a_{i+1}} < p_{a_{i+2}} < R(I_{i+2})$. Furthermore, the pillars of $\mathcal{P}$ are placed in such a way that each arch contains at least one interval of $\intI_H$, so certainly there is at least one interval $I'$ between $p_{a_{i+1}}$ and $p_{a_{i+2}}$. Then $I_i, I_{i+2}, I'$ satisfy the conditions of Lemma~\ref{lem:nothree}, and as above we find a clique in $H$ of size $\omega(H)+1$, a contradiction.
    \end{claimproof}

    So, the set $S'$ defined above does not contain a chain of size $4\omega(H)+1$. We wish to apply Lemma~\ref{lem:strongposet}, but we should be careful about whether the conditions of the lemma hold. We start first by assuming so.

    \begin{claim}\label{claim:lemmaholds}

    Suppose that the conditions to apply Lemma~\ref{lem:strongposet} to $S'$ and $P_{n,m}$ are met. Then $n \leq 10\omega(H)^2$.
        
    \end{claim}

    \begin{claimproof}[Proof of Claim~\ref{claim:lemmaholds}]
    
    Noting that $m \leq \deg_{\mathcal{P}}(K)$ since each pillar in $\mathcal{K}$ contributes at least $1$ to $\deg_{\mathcal{P}}(K)$, and using the third assumption in the theorem, it follows that  \begin{align*}
        |S'| &\leq 4\omega(H)(n+m-4\omega(H))\\
         &\leq 4\omega(H)(n+\deg_{\mathcal{P}}(K)-4\omega(H)) \\
         &\leq 4\omega(H)(n+15\omega(H)^2-4\omega(H)).
    \end{align*}  \\  Therefore we have, using that each interval has degree at least $10\omega(H)^2$ by construction,
    \begin{align*}
    10n\omega(H)^2 &\leq S\\ &\leq \omega(H)|S'| \\ &\leq \omega(H)4\omega(H)(n+15\omega(H)^2-4\omega(H)) \\ &\leq 4\omega(H)^2(n+15\omega(H)^2) \\&= 4n\omega(H)^2+60\omega(H)^4
    \end{align*}
    from which we see that $n \leq 10\omega(H)^2$. 
    \end{claimproof}

    We want the desired inequality to hold no matter what, so we also must consider the case when Lemma \ref{lem:strongposet} does not hold.

    \begin{claim}\label{claim:lemmadoesnotholdstep1}

    Suppose that the conditions to apply Lemma \ref{lem:strongposet} to $T$ and $P_{n,m}$ are not met. Then either $n \leq 15\omega(H)^2$ or $\deg_{\mathcal{P}}(K) \leq 4\omega(H)^2$. 
        
    \end{claim}

    \begin{claimproof}[Proof of Claim~\ref{claim:lemmadoesnotholdstep1}]

    The conditions of Lemma~\ref{lem:strongposet} are not met only if $4\omega(H)+1 > \min(n,m)$, meaning either $4\omega(H) \geq n$ or $4\omega \geq m$. In the former case we have $n \leq 4\omega(H) \leq 15\omega(H)^2$. In the latter case, since each pillar admits at most $\omega(H)$ different colours among its intervals, note that then $\deg_{\mathcal{P}}(K) \leq 4\omega(H)^2$.
    \end{claimproof}

    Now, we simply need to dispense with the simple case of $\deg_{\mathcal{P}}(K) \leq 4\omega(H)^2$.

    \begin{claim}\label{claim:lemmadoesnotholdstep2}

    Suppose that $\deg_{\mathcal{P}}(K) \leq 4\omega(H)^2$. Then Theorem \ref{thm:mainresult} holds.
        
    \end{claim}

    \begin{claimproof}[Proof of Claim~\ref{claim:lemmadoesnotholdstep2}]
We form a pillar system $\mathcal{P}'$ extending $\mathcal{P}$ by placing an arbitrary new pillar between two disjoint intervals in $K$, and noting that all desired properties hold, since $\chi^*(\mathcal{P}') \leq \max(\chi^*(\mathcal{P}),10\omega(H)^2+\omega(H)) \leq 15\omega(H)^2+2\omega(H)\log_2(\omega(H))+5\omega(H)$, and both newly formed arches have degree at most $4\omega(H)^2 + \omega(H)\leq 5\omega(H)^2$.    
    \end{claimproof}

    Given the previous three claims, we may assume from now on that $n \leq 15\omega(H)^2$.

    \begin{claim}\label{claim:propertyac}

    There exists a pillar system $\mathcal{P}'$ extending $\mathcal{P}$ via the new pillars $q_1, \dots, q_n$ that satisfies properties (a) and (c) of the statement of Theorem \ref{thm:mainresult}.
        
    \end{claim}

    \begin{claimproof}[Proof of Claim~\ref{claim:propertyac}]
    
    We may apply Lemma \ref{lem:inductpillars} to construct a pillar system $\mathcal{P}'$ extending $\mathcal{P}$ via the new pillars such that 
    \begin{align*}
    \chi^*(\mathcal{P}') &\leq \max(\chi(\mathcal{P}),\deg_{\mathcal{P}}(K)+\omega(H)\lceil\log_2(n+1)\rceil) \\
    &\leq \max(15\omega(H)^2+2\omega(H)\log_2(\omega(H))+5\omega(H),15\omega(H)^2+\omega(H)\lceil\log_2(15\omega(H)^2+1)\rceil). \\ 
    \end{align*}
    This new pillar system satisfies properties (a) and (c) of the theorem statement since
    \begin{align*}
        15\omega(H)^2+\omega(H)\lceil\log_2(15\omega(H)^2+1)\rceil)  &\leq
    15\omega(H)^2+\omega(H)(\log_2(15\omega(H)^2+1)+1) \\ &\leq 
    15\omega(H)^2+\omega(H)(\log_2(16\omega(H)^2)+1) \\ &\leq 
    15\omega(H)^2+2\omega(H)\log_2(\omega(H))+5\omega(H) 
    \end{align*}
    and for each new arch $K'$ of $\mathcal{P'}$ containing at least two disjoint intervals, we have
    \begin{align*}
        \deg_{\mathcal{P}'}(K') &\leq 10\omega(H)^2+\omega(H)\lceil\log_2(15\omega(H)^2+1)\rceil \\
        &\leq 10\omega(H)^2+\omega(H)(\log_2(15\omega(H)^2+1)+1) \\
        &\leq 10\omega(H)^2+\omega(H)(\log_2(16\omega(H)^2)+1) \\
        &= 10\omega(H)^2+2\omega(H)\log_2(\omega(H))+5\omega(H) \\
        &\leq 15\omega(H)^2
    \end{align*}
    where the final inequality can be seen to hold since the derivative of the difference $5\omega(H)^2-2\omega(H)\log_2(\omega(H))-5\omega(H)$ is positive for $\omega(H) \geq 1$, and the inequality holds at $\omega(H) = 1$.
    \end{claimproof}
    
    Thus, we have a pillar system $\mathcal{P}'$ extending $\mathcal{P}$ and satisfying properties (a) and (c) of the theorem statement. We must check if it satisfies property (b), or can be modified to satisfy property (b). 
    \begin{claim}\label{claim:propetybpart1}

    Let $\mathcal{P}'$ be the pillar system extending $\mathcal{P}$ by $q_1, \dots, q_n$ and satisfying properties (a) and (c) of the statement of Theorem \ref{thm:mainresult} that exists by Claim \ref{claim:propertyac}. Then one of the following holds:
    \begin{itemize}
        \item $\mathcal{P}'$ also satisfies property (b).
        \item $\mathcal{P}' \backslash \{q_n\}$ satisfies properties (a), (b), and (c).
        \item The interval $(q_n, R(K))$ contains no elements of $\intI_H$, and the interval $(q_{n-1},R(K))$ contains two disjoint elements of $\intI_H$.
    \end{itemize}
        
    \end{claim}

    \begin{claimproof}[Proof of Claim~\ref{claim:propetybpart1}]
    
    It is clear from construction that $\mathcal{P}'$ satisfies property (b) unless $(q_n,R(K))$ does not contain an interval of $\intI_H$. In that case, if $(q_{n-1},R(K))$ does not contain two disjoint intervals from $\intI_H$, we may remove the pillar $q_n$ and check that the resulting construction still satisfies all of the properties (a),(b),(c). Note that $\mathcal{P'} \setminus \{q_n\}$ is non-empty in this case, since $K$ contains two disjoint intervals by assumption. 
    \end{claimproof}

    Thus, towards our goal of finding a pillar system satisfying (a), (b), and (c), we need only consider the last case of the above claim.
    \begin{claim}\label{claim:propertybpart2}
        Let $\mathcal{P}'$ be the pillar system extending $\mathcal{P}$ by $q_1, \dots, q_n$ and satisfying properties (a) and (c) of the statement of Theorem \ref{thm:mainresult} that exists by Claim \ref{claim:propertyac}. Suppose additionally that $(q_n, R(K))$ contains no intervals of $\intI_H$, and that $(q_{n-1},R(K))$ contains two disjoint intervals of $\intI_H$.

        Then there exists a pillar system $\mathcal{Q}$, formed from $\mathcal{P}'$ by moving only the pillar $q_n$ to a new location, such that $\mathcal{Q}$ satisfies properties (a), (b), and (c) of Theorem \ref{thm:mainresult}.
    \end{claim}

    \begin{claimproof}[Proof of Claim~\ref{claim:propertybpart2}]
    Consider the second-to-last arch $(q_{n-1},q_n)$. If $(q_{n-1},q_n)$ itself contains at least two intervals of $\intI_H$, then $\deg_{\mathcal{P}}((q_{n-1},q_n)) = 10\omega(H)^2$ by construction. Thus, we may shift $q_n$ to the left until $(q_n,R(K))$ contains exactly one interval, and call the resulting pillar system $\mathcal{Q}$. Since there are two disjoint intervals in $(q_{n-1},R(K))$, there remains an interval in $(q_{n-1},q_n)$, so the system $\mathcal{Q}$ clearly satisfies properties (b) and (c); furthermore, shifting this pillar does not affect the portion of Lemma \ref{lem:inductpillars} relating to chromatic number (only the number of pillars matters), so this new system satisfies (a) as well.

    Now we assume that $(q_{n-1},q_n)$ does not contain at least two intervals. By construction, $(q_{n-1},q_n)$ then contains exactly one interval $I_N$; by assumption, there also exists at least one interval with one endpoint contained in $(q_{n-1},q_n)$, and one endpoint in $(q_n,R(K))$. Let $I'$ be the one of these intervals with rightmost left endpoint. Slide $q_n$ left to $q_n' = L(I')$, and call the resulting pillar system (with $q_n'$ instead of $q_n$) $\mathcal{Q}$. Then the resulting interval $(q_{n-1}, q_n')$ contains exactly one interval (namely $I_N$), as otherwise $(q_{n-1},R(K))$ does not contain two disjoint intervals. Moreover, $(q_n,R(K))$ now contains exactly one interval since it previously contained none and by the choice of $I'$ by rightmost left endpoint. Thus, properties (b) and (c) are satisfied by $\mathcal{Q}$, and as above, the bound on chromatic number is unchanged, so (a) is also satisfied.
    \end{claimproof}

    Finally, we need only check that whichever pillar system we output does in fact place a pillar in between two disjoint intervals of $\intI_H$.
    \begin{claim}\label{claim:finalclaim}

    Let $\mathcal{Q}$ be a pillar system extending $\mathcal{P}$ and satisfying properties (a), (b), and (c), as exists by the previous claims. 

    Then either $\mathcal{Q}$ satisfies the statement of Theorem \ref{thm:mainresult}, or we may construct a pillar system that does.
        
    \end{claim}

    \begin{claimproof}[Proof of Claim~\ref{claim:finalclaim}]

    By construction, as long as at least one pillar is placed by the algorithm, we also split a pair of disjoint intervals of $K$ into different arches, so the only remaining case to consider is if no pillars are placed at the start of the algorithm. This happens only when $\deg_{\mathcal{P}}(K) < 10\omega(H)^2$, but we have already established that in this case the theorem holds by Claim \ref{claim:lemmadoesnotholdstep2}.
    \end{claimproof}

    This concludes the overall proof of Theorem \ref{thm:mainresult}. 
\end{proof}

Now we simply need to wrap things up.

\begin{cor}\label{cor:chinumber}
    Let $H$ be a crossing graph. Then 
     $$\chi(H) \leq 15\omega(H)^2+2\omega(H)\log_2(\omega(H))+6\omega(H).$$
\end{cor}

\begin{proof}
    By repeatedly applying Theorem \ref{thm:mainresult}, we terminate in a pillar system $\mathcal{P}$ for $H$ such that $\chi(\mathcal{P}) \leq 15\omega(H)^2+2\omega(H)\log_2(\omega(H))+5\omega(H)$, and all remaining uncoloured intervals may be partitioned into sets $A_1, \dots, A_k$ so that for each $i$, all intervals in $A_i$ are contained in the same arch (and all intervals contained in this arch are in $A_i$). By construction, each interval of $A_i$ is non-adjacent to each interval of $A_j$ for $i \neq j$, and additionally, since no arch contains two disjoint intervals, each $A_i$ consists of intervals that all contain some common point. Therefore, we may select a set of $\omega$ colours not previously used, and use this same set to colour the intervals of each $A_i$ independently by Lemma \ref{lem:compgraph}, and we are done.
\end{proof}

\subsection{Main Result}

Taking together everything in the previous section, we may now prove our main result.

\newtheorem*{thm:finaltheorem}{Theorem \ref{thm:finaltheorem}}
\begin{thm:finaltheorem}

Let $T$ be a crossing tournament. Then
$$
\vchi(T) \leq 15\vomega(T)^4+2\vomega(T)^3\log_2(\vomega(T))+6\vomega(T)^3.
$$

\end{thm:finaltheorem}

\begin{proof}
    Letting $G'$ be a backedge graph of $T$ realizing the clique number of $T$, by Observation \ref{obs:reduction} and Lemma \ref{lem:minchromD} we have $\vchi(T) \leq \chi(G') \leq (\vomega(T))^2f(\vomega(T))$, where $f$ is a $\chi$-bounding function for crossing graphs.  
    Now let $\vomega(T) = \omega$. By Corollary \ref{cor:chinumber}, one such function is $f(\omega) =  15\omega^2+2\omega\log_2(\omega)+6\omega$. Plugging this into the above yields the desired result.
\end{proof}

\section{Chordal Backedge graphs}
\label{sec:chordal}
In this section we prove Theorem \ref{thm:chordalexample}. The idea is to use the fact that chordal graphs have a perfect elimination ordering which can be represented as a tree structure, and that Burling graphs, which are triangle-free graphs with large chromatic number, also have a tree representation. Recall that chordal graphs are graphs such that every cycle of length at least four is not induced. Another useful equivalent way of thinking about chordal graphs is that they are precisely the graphs which admit \textbf{perfect elimination orderings} \cite{FulkersonGross1965}. Given a graph $G$, a linear ordering $v_{1},\ldots,v_{n}$ of $V(G)$ is a \textbf{perfect elimination ordering} of $G$ if, for every $i \in \{1,\ldots,n\}$, the set $\{v_{j} : j >i \text{ and } v_{i}v_{j} \in E(G)\}$ is a clique. 
We will need the definition of Burling Graphs, and in particular the Burling Tree definition from \cite{BurlingGraphRevisted1,BurlingGraphRevisited2}.

\begin{definition}
Given a rooted tree $(T, r)$, a \emph{branch} in $T$ is a path $v_1, \dots, v_k$ such that for all $i \in \{1, \dots, k-1\}$, $v_i$ is the parent of $v_{i+1}$. 

A \textbf{Burling tree} is a $4$-tuple $(T,r,\ell, c)$ in which

\begin{itemize}
    \item $T$ is a rooted tree and $r$ is its root,
    \item $\ell$ is a function associating to each vertex $v$ of $T$ which is not a leaf, one child of $v$ which is called the \textbf{last-born} of $v$,
    \item $c$ is a function defined on the vertices of $T$. If $v$ is a non-last-born vertex in $T$ other than the root, then $c$ associates to $v$ the vertex set of a (possibly empty) branch in $T$ starting at the last-born of the parent of $v$. If $v$ is a last-born or the root of $T$, then we define $c(v) = \emptyset$. We call $c$ the \textit{choose} function of $T$.
\end{itemize}
\end{definition}
We say the oriented graph $G$ is \textbf{fully derived} from the Burling tree $(T,r,\ell,c)$ if it is the oriented graph whose vertex set is $V(T)$ and $uv \in A(G)$ if and only if $v$ is a vertex in $c(u)$.

A non-oriented graph $G$ is \textbf{fully derived} from $(T,r,\ell,c)$ if it is the underlying graph of the oriented graph fully derived from $T$. 

A graph (respectively, oriented graph) $G$ is \textbf{derived} from a Burling tree $(T,r,\ell,c)$ if it is an induced subgraph of a graph (respectively, oriented graph) fully derived from $T$. An oriented or non-oriented graph $G$ is called a \textit{derived graph} if there exists a Burling tree $(T,r,\ell,c)$ such that $G$ is derived from $(T,r,\ell,c)$. 

As shown in~\cite{BurlingGraphRevisted1}, a graph is a Burling graph if and only if it is a derived graph\footnote{Originally, Burling Graphs were defined through axis-aligned box intersection graphs \cite{Burling1965}.} We need a classical fact that Burling graphs are triangle-free, and have arbitrarily large chromatic number.

\begin{theorem}[\cite{Burling1965}]
    For every positive integer $k$, there is a Burling Tree $(T_{k},r,\ell,c)$ such that the fully derived graph $G$ from $(T_{k},r,\ell,c)$ is triangle-free and $\chi(G) \geq k$. 
\end{theorem}

We are now ready to show that there exist tournaments with both a Burling graph and a chordal graph as backedge graphs.

\begin{lemma}
\label{lem:burling}
    Let $k$ be any positive integer. Let $(T,r,\ell,c)$ be a Burling Tree with fully derived graph $G$ such that $G$ is triangle-free and has chromatic number at least $k$. Then there exists a tournament $Q$ and two orders of $V(Q)$, $<_{1}$ and $<_{2}$, such that $B(Q,<_{1})$ is isomorphic to $G$, and $B(Q,<_{2})$ is isomorphic to a chordal graph $C$ where $<_{2}$ is a perfect elimination ordering of $C$. 
\end{lemma}

\begin{proof}
    We construct the two orderings as follows. First, at each internal vertex $v \in V(T)$, fix an ordering $<$ of the children of $v$ such that the last-born vertex is last in this ordering. Now let $<_{1}$ be the ordering of vertices of $V(T)$ which is a depth first search tree starting at $r$ where at every vertex $v$, we search the children of $v$ in the order of $<$. Let $<_{2}$ be obtained from the same DFS search tree of $<_{1}$, except instead of recording the first time we reach a vertex, we record the last time we visit the vertex in the search tree. 

    Observe that $u <_{1} v$ and $v <_{2} u$ if and only if  either $u$ lies on the path from $v$ to $r$, or $v$ lies on the path from $u$ to $r$. This follows since if $u$ and $v$ are not on the same path to $r$, then if a DFS search first visits $u$, it will visit $u$ for the last time before visiting $v$, and vice versa. If $u$ lies on the path from $v$ to $r$, then $u <_{1} v$ by construction, but $v <_{2} u$ as the last visit to $u$ is after the last visit to $v$. 

    Now let $C$ be the graph obtained from $G$ by adding an edge between every pair of vertices $u,v$ such that $u<_{1} v$ and $v <_{2} u$, or $u <_{2} v$ and $v <_{1} u$; that is, $u$ and $v$ such that there exists an ancestral relationship between $u$ and $v$.  Let $Q$ be the tournament such that $B(Q,<_{1})$ is isomorphic to $G$. We claim that $C$ is chordal, $B(Q,<_{2}) = C$, and that $<_{2}$ is a perfect elimination ordering for $C$.

    First let us show that $B(Q,<_{2}) = C$. Note that if $uv \in E(G)$, then $u$ and $v$ do not lie on the same path to $r$: If we assume without loss of generality that $u$ is not a last-born, then $v$ lies in a branch starting at the last-born of the parent of $u$. Thus, for every edge of $G$, its ends have the same relative order in $<_1$ and $<_2$, and so every edge of $G$ is an edge of $C$. Now we only need to consider the case when $uv \in E(C)$ but $uv \not \in E(G)$. By construction, this implies that $u <_{1} v$ and $v <_{2} u$, or $u <_{2} v$ and $v <_{1} u$.  If $u <_{1} v$, $v<_{2} u$ and $uv \not \in E(G)$, then $vu$ is an arc of $Q$, which implies that $vu \in B(Q,<_{2})$. The other case is symmetric. 

    Thus all that is left to do is show that $C$ is chordal and that $<_{2}$ is a perfect elimination ordering. Fix a vertex $v \in V(C)$, and let $N = \{u \in N_C(v): v <_2 u\}$ be the set of vertices that are neighbours of $v$, and also after $v$ under $<_{2}$. We want to show that $N$ induces a clique. 

    Observe that by construction, no descendant of $v$ in $T$ is after $v$ in the ordering $<_2$. If $vx \in E(G)$, and $v <_{2} x$ then by construction, $x$ is in $c(v)$, and thus $v$ is a non-last-born vertex in this case. In other words, $N \cap N_G(v) = c(v).$ In particular, if $v$ is last-born, then all neighbours obtained from $G$ precede $v$. It remains to consider $N' = N \setminus N_G(v)$. Since $N'$ contains no descendant of $v$, and from the definition of $C$, it follows that $N'$ is the set of ancestors of $v$ in $T$. It follows that $N$ induces a branch in $T$, starting from $r$, to the parent of $v$, to the last-born sibling of $v$, and then along $c(v)$.   This implies that there is an ancestral relationship between every pair of vertices in $N$.  Therefore $N$ is a clique, and thus $<_{2}$ is a perfect elimination ordering. It follows that $C$ is chordal.   
\end{proof}

We are almost ready to show tournaments with chordal backedge graphs are not $\vec{\chi}$-bounded. We need a lemma from \cite{NGUYEN2025146}.

\begin{lemma}[\cite{NGUYEN2025146}]
\label{lem:boundlemma}
For any tournament $T$ and ordering $<$ of $V(T)$,
\[
\frac{\chi(B(T,<))}{\omega(B(T,<))} \le \vec{\chi}(T) \le \chi(B(T,<)).
\]
\end{lemma}

Now we show that the class of tournaments with chordal backedge graphs are not $\vec{\chi}$-bounded.

By Lemma \ref{lem:burling}, for every positive integer $k$, there exists a tournament $T$ such that there is an ordering $<_{1}$ of $T$ such that $B(T,<_{1}) = B$ where $B$ is a triangle-free graph with chromatic number at least $k$, and further there is an ordering $<_{2}$ such that $B(T,<_{2}) = C$ where $C$ is a chordal graph. Thus the clique number of $T$ is at most $2$, as seen by $B$, and $\chi(B) \geq k$, thus by Lemma \ref{lem:boundlemma} we have that $\vec{\chi}(T) \geq \frac{k}{2}$. This yields a construction of tournaments with clique number at most 2 and unbounded chromatic number. However, each tournament we construct has a chordal backedge graph, and thus the class of tournaments with chordal backedge graphs is not $\vec{\chi}$-bounded.

\bibliographystyle{plain}
\bibliography{bib}

@article{davies,
  title={Improved bounds for colouring circle graphs},
  author={Davies, James},
  journal={Proceedings of the American Mathematical Society},
  volume={150},
  number={12},
  pages={5121--5135},
  year={2022}
}

@book{brandstadt1999graph,
  title={Graph classes: a survey},
  author={Brandst{\"a}dt, Andreas and Le, Van Bang and Spinrad, Jeremy P},
  year={1999},
  publisher={SIAM}
}

@article{gyarfas1987problems,
  title={Problems from the world surrounding perfect graphs},
  author={Gy{\'a}rf{\'a}s, Andr{\'a}s},
  journal={Applicationes Mathematicae},
  volume={19},
  number={3-4},
  pages={413--441},
  year={1987},
  publisher={Polska Akademia Nauk. Instytut Matematyczny PAN}
}

@article{perfectIntoComparability,
  title={Partitioning perfect graphs into comparability graphs},
  author={Gy{\'a}rf{\'a}s, Andr{\'a}s and Marits, M{\'a}rton and T{\'o}th, G{\'e}za},
  journal={arXiv preprint, arXiv:2408.13523},
  year={2024}
}

@article{crossing,
  title={Some results and problems on tournament structure},
  author={Nguyen, Tung and Scott, Alex and Seymour, Paul},
  journal={Journal of Combinatorial Theory, Series B},
  volume={173},
  pages={146--183},
  year={2025},
  publisher={Elsevier}
}

@article{original,
  title={Clique number of tournaments},
  author={Aboulker, Pierre and Aubian, Guillaume and Charbit, Pierre and Lopes, Raul},
  journal={arXiv preprint, arXiv:2310.04265},
  year={2023}
}

@book{diestel,
  title={Graph theory},
  author={Diestel, Reinhard},
  volume={173},
  year={2025},
  publisher={Springer Nature}
}

@article{neumann,
  title={The dichromatic number of a digraph},
  author={Neumann-Lara, Victor},
  journal={Journal of Combinatorial Theory, Series B},
  volume={33},
  number={3},
  pages={265--270},
  year={1982},
  publisher={Elsevier}
}

@article{separating,
  title={Separating polynomial $\chi$-boundedness from $\chi$-boundedness},
  author={Bria{\'n}ski, Marcin and Davies, James and Walczak, Bartosz},
  journal={Combinatorica},
  volume={44},
  number={1},
  pages={1--8},
  year={2024},
  publisher={Springer}
}

@article{strongperfect,
  title={The strong perfect graph theorem},
  author={Chudnovsky, Maria and Robertson, Neil and Seymour, Paul and Thomas, Robin},
  journal={Annals of mathematics},
  pages={51--229},
  year={2006},
  publisher={JSTOR}
}

@article{substitution,
  title={Substitution and $\chi$-boundedness},
  author={Chudnovsky, Maria and Penev, Irena and Scott, Alex and Trotignon, Nicolas},
  journal={Journal of Combinatorial Theory, Series B},
  volume={103},
  number={5},
  pages={567--586},
  year={2013},
  publisher={Elsevier}
}

@inproceedings{berge,
  title={Les problemes de coloration en th{\'e}orie des graphes},
  author={Berge, Claude},
  booktitle={Annales de l'ISUP},
  volume={9},
  number={2},
  pages={123--160},
  year={1960}
}

@article{gyarfasCircle,
  title={On the chromatic number of multiple interval graphs and overlap graphs},
  author={Gy{\'a}rf{\'a}s, Andr{\'a}s},
  journal={Discrete mathematics},
  volume={55},
  number={2},
  pages={161--166},
  year={1985},
  publisher={Elsevier}
}

@article{mccarty,
  title={Circle graphs are quadratically $\chi$-bounded.},
  author={Davies, James and McCarty, Rose},
  journal={Bulletin of the London Mathematical Society},
  volume={53},
  number={3},
  year={2021}
}

@article{kostochka,
  title={Upper bounds on the chromatic number of graphs},
  author={Kostochka, Alexandr V},
  journal={Trudy Inst. Mat.(Novosibirsk)},
  volume={10},
  number={Modeli i Metody Optim.},
  pages={204--226},
  year={1988}
}

@article{cliquewidth,
	author = {Bonamy, Marthe and Pilipczuk, Micha\l{}},
	journal = {Advances in Combinatorics},
	doi = {10.19086/aic.13668},
	year = {2020},
	month = {jul 10},
	title = {Graphs of bounded cliquewidth are polynomially $\chi$-bounded},
}

@article{Lgraphs,
title = {On bounding the chromatic number of L-graphs},
journal = {Discrete Mathematics},
volume = {154},
number = {1},
pages = {179-187},
year = {1996},
issn = {0012-365X},
doi = {https://doi.org/10.1016/0012-365X(95)00316-O},
url = {https://www.sciencedirect.com/science/article/pii/0012365X9500316O},
author = {Sean McGuinness},
}

@article{GroundedLgraphs,
author = {Davies, James and Krawczyk, Tomasz and McCarty, Rose and Walczak, Bartosz},
title = {Grounded L-Graphs Are Polynomially $\chi$-Bounded},
year = {2023},
issue_date = {Dec 2023},
publisher = {Springer-Verlag},
address = {Berlin, Heidelberg},
volume = {70},
number = {4},
issn = {0179-5376},
url = {https://doi.org/10.1007/s00454-023-00592-z},
doi = {10.1007/s00454-023-00592-z},
journal = {Discrete Comput. Geom.},
month = nov,
pages = {1523–1550},
numpages = {28},
}

@article{bousquetspaper,
author = {Bousquet, Nicolas and Thomass\'{e}, St\'{e}phan},
title = {Scott's induced subdivision conjecture for maximal triangle-free graphs},
year = {2012},
issue_date = {July 2012},
publisher = {Cambridge University Press},
address = {USA},
volume = {21},
number = {4},
issn = {0963-5483},
url = {https://doi.org/10.1017/S0963548312000065},
doi = {10.1017/S0963548312000065},
journal = {Comb. Probab. Comput.},
month = jul,
pages = {512–514},
numpages = {3}
}

@article{survey,
author = {Scott, Alex and Seymour, Paul},
title = {A survey of $\chi$-boundedness},
journal = {Journal of Graph Theory},
volume = {95},
number = {3},
pages = {473-504},
doi = {https://doi.org/10.1002/jgt.22601},
url = {https://onlinelibrary.wiley.com/doi/abs/10.1002/jgt.22601},
eprint = {https://onlinelibrary.wiley.com/doi/pdf/10.1002/jgt.22601},
year = {2020}
}

@article{Nesetril2020ClusteringPO,
  title={Clustering powers of sparse graphs},
  author={Jaroslav Nesetril and Patrice Ossona de Mendez and Michał Pilipczuk and Xuding Zhu},
  journal={Electron. J. Comb.},
  year={2020},
  volume={27},
  pages={4},
  url={https://api.semanticscholar.org/CorpusID:212634020}
}

@article{lindapaper,
title = {Reuniting $\chi$-boundedness with polynomial $\chi$-boundedness},
journal = {Journal of Combinatorial Theory, Series B},
volume = {176},
pages = {30-73},
year = {2026},
issn = {0095-8956},
doi = {https://doi.org/10.1016/j.jctb.2025.08.002},
url = {https://www.sciencedirect.com/science/article/pii/S0095895625000589},
author = {Maria Chudnovsky and Linda Cook and James Davies and Sang-il Oum},
}

@article{nolongholevtxminor,
title = {Classes of graphs with no long cycle as a vertex-minor are polynomially $\chi$-bounded},
journal = {Journal of Combinatorial Theory, Series B},
volume = {140},
pages = {372-386},
year = {2020},
issn = {0095-8956},
doi = {https://doi.org/10.1016/j.jctb.2019.06.001},
url = {https://www.sciencedirect.com/science/article/pii/S0095895619300590},
author = {Ringi Kim and O-joung Kwon and Sang-il Oum and Vaidy Sivaraman},
}

@article{bourneuf2025bounded,
  title={Bounded twin-width graphs are polynomially $\chi$-bounded},
  author={Bourneuf, Romain and Thomass{\'e}, St{\'e}phan},
  journal={Advances in Combinatorics},
  year={2025}
}

@article{BERGER20131,
title = {Tournaments and colouring},
journal = {Journal of Combinatorial Theory, Series B},
volume = {103},
number = {1},
pages = {1-20},
year = {2013},
issn = {0095-8956},
doi = {10.1016/j.jctb.2012.08.003},
url = {https://doi.org/10.1016/j.jctb.2012.08.003},
author = {Eli Berger and Krzysztof Choromanski and Maria Chudnovsky and Jacob Fox and Martin Loebl and Alex Scott and Paul Seymour and Stéphan Thomassé},
}

@article{bigNote,
  title={Decomposing tournaments into comparability graphs},
  author={Pierre Aboulker and Logan Crew and Julien Duron and Xinyue Fan and Hugo Jacob and Rémy Kimbrough and Hidde Koerts and Benjamin Moore and Sophie Spirkl },
  journal={ar{X}iv:2606.07748},
  year={2026}
}

@article{BurlingGraphRevisited2,
title = {Burling graphs revisited, part {II}: Structure},
journal = {European Journal of Combinatorics},
volume = {116},
pages = {103849},
year = {2024},
issn = {0195-6698},
doi = {https://doi.org/10.1016/j.ejc.2023.103849},
url = {https://www.sciencedirect.com/science/article/pii/S0195669823001671},
author = {Pegah Pournajafi and Nicolas Trotignon},
}

@article{BurlingGraphRevisted1,
title = {Burling graphs revisited, part {I}: New characterizations},
journal = {European Journal of Combinatorics},
volume = {110},
pages = {103686},
year = {2023},
issn = {0195-6698},
doi = {https://doi.org/10.1016/j.ejc.2023.103686},
url = {https://www.sciencedirect.com/science/article/pii/S0195669823000033},
author = {Pegah Pournajafi and Nicolas Trotignon},
}

@phdthesis{Burling1965,
  author = {Burling, James Perkins},
  title = {On coloring problems of families of polytopes},
  school = {University of Colorado},
  address = {Boulder, Colorado},
  year = {1965}
}

@article{NGUYEN2025146,
title = {Some results and problems on tournament structure},
journal = {Journal of Combinatorial Theory, Series B},
volume = {173},
pages = {146-183},
year = {2025},
issn = {0095-8956},
doi = {https://doi.org/10.1016/j.jctb.2025.02.002},
url = {https://www.sciencedirect.com/science/article/pii/S0095895625000097},
author = {Tung Nguyen and Alex Scott and Paul Seymour},

}

@article{FulkersonGross1965,
  author  = {Fulkerson, D. R. and Gross, O. A.},
  title   = {Incidence matrices and interval graphs},
  journal = {Pacific Journal of Mathematics},
  volume  = {15},
  number  = {3},
  pages   = {835--855},
  year    = {1965}
}

\end{document}